\documentclass{article}
\usepackage{graphicx,tikz}
\usetikzlibrary{decorations.markings,decorations.pathreplacing}
\usepackage[margin=1.25in]{geometry}
\usepackage{amsfonts,amsmath,amssymb,amsthm}

\newtheorem{theorem}{Theorem}
\newtheorem{lemma}[theorem]{Lemma}

\newtheorem{proposition}[theorem]{Proposition}
\newtheorem{conjecture}[theorem]{Conjecture}

\newtheorem{claim}[theorem]{Claim}

\theoremstyle{definition}

\newcommand{\EE}{\mathbb E}
\renewcommand{\Pr}{\mathbb P}
\newcommand{\T}{\overrightarrow{T}}
\newcommand{\Tzero}{\overrightarrow{T_0}}
\newcommand{\Tstar}{\overrightarrow{T^*}}
\newcommand{\Astar}{\overrightarrow{A^*}}
\newcommand{\Bstar}{\overrightarrow{B^*}}
\newcommand{\A}{\overrightarrow{A}}
\newcommand{\Azero}{\overrightarrow{A_0}}
\newcommand{\B}{\overrightarrow{B}}
\renewcommand{\le}{\leqslant}
\renewcommand{\ge}{\geqslant}

\newcommand{\bb}[1]{\bigl(#1\bigr)}
\newcommand{\cF}{\mathcal{F}}

\newcommand{\red}[1]{{\color{red}#1}}

\title{A short proof of the Erd\H os--S\'os Conjecture}
\author{Oliver Riordan\footnote{Mathematical Institute, University of Oxford.
\texttt{\{oliver.riordan,alexander.scott\}@maths.ox.ac.uk} \newline Alex Scott's research supported by EPSRC grant EP/X013642/1.}\;  and
Alex Scott\protect\footnotemark[1]
}

\date{8 October 2026}

\begin{document}

\maketitle

\begin{abstract}
The Erd\H os--S\'os Conjecture was recently proved by GPT-6 Astra, using a very ingenious and surprising argument.  In this note, we present a simplified version of this argument in an (arguably) more natural form.  We also prove a related conjecture of Addario-Berry, Havet, Linhares Sales, Reed and Thomass\'e, and give a substantially simplified proof of an analogous result for eulerian digraphs, again proved by GPT-6 Astra, and announced by Mubayi and Verstraete.
In addition, we determine the extremal (di)graphs for all three problems.
\end{abstract}

In the 1960s, Erd\H os and S\'os made the following beautiful conjecture concerning the extremal numbers of arbitrary trees:

\begin{conjecture}\label{ESconj}(Erd\H os--S\'os Conjecture \cite{erdossos})
    Let $G$ be a graph with average degree greater than $k-2$.  Then $G$ contains a copy of every tree $T$ on $k$ vertices.
\end{conjecture}

The bound is best possible for any $T$, as shown by a vertex-disjoint union of complete graphs $K_{k-1}$. Over the last 60 years, the Erd\H os--S\'os Conjecture attracted considerable attention. Special cases were proved by, for example, Brandt and Dobson~\cite{BD}, McLennan~\cite{McLennan}, and Pokrovskiy~\cite{pokrovskiy}; for a more detailed history see the survey by Stein~\cite{Stein}. The conjecture was proved in a certain asymptotic sense by Rohzoň~\cite{rohzon}, and independently by Besomi, Pavez-Sign\'e and Stein~\cite{BP-SS}. 

In a startling development, the Erd\H os--S\'os conjecture was recently proved in full by GPT-6 Astra~\cite{chatgpt}, with a short and ingenious argument; there is also a Lean verification \cite{lean}.  An exposition of the proof was posted by Thomas Bloom \cite{bloom}.

The first aim of this note is to present a simplified and hopefully more intuitive version of GPT-6 Astra's proof of the Erd\H os--S\'os Conjecture.  In addition, we determine the extremal graphs, that is the graphs that have average degree $|T|-2$ but do not contain a copy of $T$.

\begin{theorem}\label{th:extremal}
    Let $T$ be a tree with $k$ vertices, and let $G$ be a graph with average degree $k-2$.  If $G$ does not contain a copy of $T$ then:
\begin{itemize}
    \item If $T=K_{1,k-1}$, then $G$ is $(k-2)$-regular.
    \item If $T \ne K_{1,k-1}$, then $G$ is a vertex-disjoint union of copies of $K_{k-1}$. 
\end{itemize}
In both cases, every graph of the given form is extremal.
\end{theorem}

We further show that with very minor modifications, the proof of the Erd\H os--S\'os Conjecture can be used to prove a related conjecture for digraphs.  An orientation of a graph is {\em antidirected} if it does not contain an oriented path with two edges; in other words, every non-isolated vertex is either a source or a sink.  We prove the following conjecture of Addario-Berry, Havet, Linhares Sales, Reed and Thomass\'e.\footnote{As we were writing this, we discovered that the same result was proved independently and a few days earlier by Santos, Stein and Williams \cite{maya} (see their Section 6), also by modifying the ChatGPT proof.}

\begin{conjecture}\label{AHLRTconj}(Addario-Berry, Havet, Linhares Sales, Reed and Thomass\'e \cite{directed})
    Let $D$ be a digraph with average outdegree greater than $k-2$.  Then $D$ contains a copy of every antidirected tree with $k$ vertices.
\end{conjecture}

Note that this is a generalization of the Erd\H os--S\'os Conjecture, as given a graph $G$ we can consider the digraph formed by replacing every edge $xy$ with a pair of directed edges, one in each direction. 
The antidirected condition is necessary: we can avoid any tree that is not antidirected by taking $D$ to be an arbitrarily large complete bipartite graph, with all edges directed from the first to the second class.
For recent progress on this conjecture, see, for example,  \cite{SZ-G24,ST-N25,maya}.

Once again, we are able to determine the extremal digraphs.

\begin{theorem}\label{th:extremal2}
    Let $\T$ be an antidirected tree with $k$ vertices, and let $D$ be a digraph with average outdegree $k-2$.  If $D$ does not contain a copy of $\T$ then:
\begin{itemize}
        \item If $\T$ is a star and the centre of the star is a source, then $D$ is $(k-2)$-out-regular; while if the centre is a sink then $D$ is $(k-2)$-in-regular. 
    \item If $\T$ is not a star, then $D$ is a vertex-disjoint union of copies of complete digraphs on $k-1$ vertices, with every pair of vertices joined by edges in both directions. 
\end{itemize}
In both cases, every digraph of the given form is extremal.
\end{theorem}

Finally, we turn to considering {\em arbitrary} orientations of trees.
In order to get these, it is necessary to make an additional assumption about our digraph.  A digraph is {\em eulerian} if every vertex has equal in-degree and out-degree.  Mubayi and Verstraete \cite{mv} recently announced the following beautiful result, which was proved by GPT-6 Astra using a similar approach to the proof of Conjecture \ref{ESconj}.  

\begin{theorem}\label{eulerianmain}
    Let $D$ be an eulerian digraph with average outdegree greater than $k-2$.  Then $D$ contains a copy of every oriented tree with $k$ vertices.
\end{theorem}

Here we provide a substantially simplified and more accessible proof. Once again, we are able to determine the extremal digraphs. 

\begin{theorem}\label{th:extremal3}
    Let $\T$ be an oriented tree with $k$ vertices, and let $D$ be an eulerian digraph with average outdegree $k-2$.  If $D$ does not contain a copy of $\T$ then:
\begin{itemize}
    \item If $\T$ is an antidirected star, then $D$ is $(k-2)$-in- and out-regular.
    \item If $\T$ is any other star, then $D$ is obtained from a $(k-2)$-regular graph by replacing every edge with edges in both directions.
    \item If $\T$ is not a star, then $D$ is a vertex-disjoint union of copies of complete digraphs on $k-1$ vertices, with every pair of vertices joined by edges in both directions. 
\end{itemize}
In all cases, every digraph of the given form is extremal.
\end{theorem}

We note that Theorem \ref{eulerianmain} was not previously known even for directed paths.  For example, a special case of a conjecture of Bollob\'as and Scott \cite{BS} asserts that every eulerian digraph with average degree $d$ contains a directed path\footnote{In fact, they made the stronger conjecture that there is a {\em cycle} of length $\Omega(d)$: this remains open.}
of length $\Omega(d)$.  Progress on the conjecture was made by Huang, Ma, Shapira, Sudakov and Yuster \cite{huang}, Janzer, Sudakov and Tomon \cite{janzer}, and then Knierim, Larcher and Martinsson \cite{knierim}, who proved a bound of the form $\Omega(d/\log d)$. Thus it is surprising to get an optimal bound for all oriented trees.

The rest of the paper is structured as follows.
We give the (very short) proofs of the Erd\H os--S\'os Conjecture \ref{ESconj} and the conjecture of Addario-Berry, Havet, Linhares Sales, Reed and Thomass\'e \cite{directed} in Sections \ref{sec:graphs} and \ref{sec:digraphs}, followed by the proof of Theorem \ref{eulerianmain} in Section ref{sec:eulerian}.  We then turn to determining the extremal (di)graphs: Theorems \ref{th:extremal}, \ref{th:extremal2} and \ref{th:extremal3} are proved in Sections \ref{sec:extremalgraphs}, \ref{sec:extremaldigraphs} and \ref{sec:extremaleulerian} respectively.

We note that the proof of the Erd\H os--S\'os Conjecture that we give is quite close to the GPT one, but is presented very differently: we hope it is simpler and significantly easier to follow.  The proof of the conjecture of Addario-Berry, Havet, Linhares Sales, Reed and Thomass\'e follows with only minor modification.  For the eulerian case, the proof is inspired by the proof from \cite{mv}, but it is quite different and much simpler. However, it does use a beautiful lemma from that paper, for which we give our own simpler proof, which also establishes the equality case.  We strengthen all three theorems by determining the extremal (di)graphs: as far as we are aware, Theorems \ref{th:extremal}, \ref{th:extremal2} and \ref{th:extremal3} are all new results.

 \section{Graphs}\label{sec:graphs}

Let $T$ be a rooted tree.  Given a graph $G$ with an ordered vertex set $v_1<\dots<v_n$, we say that an edge $v_1v_i$ is {\em $T$-jumping} if $G$ contains a copy of $T$ rooted at $v_1$, with all its vertices in $\{v_1,\dots,v_{i-1}\}$.  

We will prove the following theorem. 

\begin{theorem}\label{th1}
    Let $G$ be a graph with average degree $d$, and let $T$ be a rooted tree.  Then the expected number of $T$-jumping edges in a uniformly random ordering of the vertices of $G$ is at least $d-|T|+1$.
\end{theorem}
 To deduce Conjecture~\ref{ESconj} for an arbitrary tree $T$, apply Theorem~\ref{th1} to a tree $T'$ obtained by removing a leaf from $T$, and rooting the resulting tree at the neighbour of the removed leaf. If $G$ contains no copy of $T$, then however its vertices are ordered, there are no $T'$-jumping edges, so the average degree is at most $|T'|-1=|T|-2$.

\begin{proof}[Proof of Theorem~\ref{th1}]
    We proceed by induction on $|T|$.  For $|T|=1$ the statement is immediate: every edge incident with $v_1$ is $T$-jumping, and the expected number of such edges is $d$.

    Now suppose that $T$ is a tree with root $r$, and we have proved the result for smaller values of $|T|$.  There are two cases, depending on the degree of $r$.

    \medskip

    \noindent{\em Case 1: $r$ has degree 1.}  In this case let $s$ be the neighbour of $r$ in $T$, and let $A$ be the tree obtained from $T$ by deleting $r$ and declaring $s$ to be the new root.
    By induction, the expected number of $A$-jumping edges is at least $d-|A|+1=d-|T|+2$.  

    Call an edge of $G$ {\em $A$-spanning} if it joins $v_1$ to some vertex $v_i$, and there is a copy of $A$ rooted at $v_i$ with all its vertices contained in $\{v_2,\dots,v_i\}$. A given edge $xy$ is $A$-jumping in an ordering $\pi$ if and only if it is $A$-spanning in the ordering $\pi'$ obtained by swapping the positions of $x$ and $y$. It follows that the probabilities (in a random ordering) that $e$ is $A$-spanning and that $e$ is $A$-jumping are equal. Summing over $e$, the expected number of $A$-spanning edges is thus at least  $d-|T|+2$.

For a fixed ordering, if $v_i<v_j$ and the edges $v_1v_i$ and $v_1v_j$ are both $A$-spanning then the edge $v_1v_j$ is $T$-jumping, since $v_1v_i$ and the copy of $A$ guaranteed by $v_1v_i$ being $A$-spanning give a copy of $T$ rooted at $v_1$ (see Fig \ref{f1}).  Thus, in every ordering, all but at most one $A$-spanning edge is $T$-jumping.  It follows that the expected number of $T$-jumping edges in a random ordering is at least $d-|T|+1$, as required.

     \begin{figure}[ht]
\centering
\begin{tikzpicture}[
  line cap=round,
  line join=round,
  every path/.style={draw=black, line width=0.9pt}
]
  % Left-hand picture
  \coordinate (Lroot) at (0,0);
  \coordinate (Linner) at (2.15,0.05);
  \coordinate (Louter) at (3.05,0.05);
  \coordinate (Ljumping) at (4.20,0.10);

  % Curved triangle A
  \draw (Lroot)
    .. controls (1.05,1.15) and (2.55,1.30) .. (Louter);
  \draw (Lroot)
    .. controls (0.72,0.55) and (1.82,0.65) .. (Linner);
  \draw (Linner) -- (Louter);
  \node at (2.55,0.35) {$A$};

  % Curved jumping edge
  \draw (Lroot)
    .. controls (1.10,1.55) and (3.50,1.60) .. (Ljumping);
  \node at (2.20,1.47) {$e$};

  \fill (Lroot) circle (1.7pt);
  \fill (Ljumping) circle (1.7pt);

  % Central arrow
  \draw[->,line width=1pt] (5.05,0.52) -- (6.05,0.52);

  % Right-hand picture
  \begin{scope}[xshift=6.90cm]
    \coordinate (Rjumping) at (0,0.10);
    \coordinate (Router) at (1.15,0.05);
    \coordinate (Rinner) at (2.05,0.05);
    \coordinate (Rroot) at (4.20,0);

    % Curved jumping edge
    \draw (Rjumping)
      .. controls (0.70,1.60) and (3.10,1.55) .. (Rroot);
    \node at (2.00,1.47) {$e$};

    % Curved triangle A
    \draw (Rroot)
      .. controls (3.15,1.15) and (1.65,1.30) .. (Router);
    \draw (Rroot)
      .. controls (3.48,0.55) and (2.38,0.65) .. (Rinner);
    \draw (Router) -- (Rinner);
    \node at (1.65,0.35) {$A$};

    \fill (Rjumping) circle (1.7pt);
    \fill (Rroot) circle (1.7pt);
  \end{scope}
\end{tikzpicture}
\caption{\label{f1} A rooted copy of $A$ with a jumping edge becomes a rooted copy of $T$ when the ends of $e$ are exchanged}
\end{figure}

    \medskip

    \noindent{\em Case 2: $r$ has degree at least 2.}  We divide $T$ into two trees $A$ and $B$, each with at least two vertices, overlapping only at the root $r$.
Given an ordering $\pi$ of the vertices of $G$, call an edge of $G$ {\em $B^*$-jumping} if it is $B$-jumping in the ordering obtained from $\pi$ by reversing the order of $v_2,\dots,v_n$ (we can think of this as looking to the left of $v_1$ in the cyclic ordering).  Let $L_\pi$ be the set of $A$-jumping edges in $\pi$,  let $R_\pi$ be the set of $B^*$-jumping edges, and let $I_\pi=L_\pi\cap R_\pi$ be the set of edges that are jumping in both directions.  Since $L_\pi$ and $R_\pi$ are subsets of a set of size $d_G(v_1)$, we have
$|I_\pi|\ge |L_\pi|+|R_\pi|-d_G(v_1)$, so for a uniformly random permutation $\pi$ it follows by induction that 
\begin{align*}
\EE|I_\pi|
&\ge \EE|L_\pi|+\EE|R_\pi|-\EE d_G(v_1)\\
&\ge (d-|A|+1)+(d-|B|+1)-d\\
&=d-|T|+1.
\end{align*}

To convert this into a bound on the number of $T$-jumping edges, we apply an involution to the set of orderings.  Given an ordering $\pi: v_1<\cdots<v_n$, let $i$ be maximal such that $G$ contains a copy of $B$ rooted at $v_1$ and with all other vertices in $\{v_i,\dots,v_n\}$; if no such $i$ exists then set $i=2$.  Let $\pi'$ be the ordering obtained from $\pi$ by concatenating $v_1$ and $v_n,v_{n-1},\dots,v_i$ and $v_2,\dots,v_{i-1}$ (thus $v_i,\dots,v_n$ reverse order and are inserted between $v_1$ and $v_2$).  This is a bijection: for the inverse map, let $i$ be minimal such that $G$ contains a copy of $B$ rooted at $v_1$ and with all other vertices in $\{v_2,\dots,v_i\}$, or take $i=n$ if there is no such $i$; we then reverse the order of $\{v_2,\dots,v_i\}$ and move them to the end.  Thus $\pi'$ is also uniformly random.

Finally, we note that if $e$ belongs to $I_\pi$, then $e$ is a $T$-jumping edge in $\pi'$. Indeed, if $e\in I_\pi$ joins $v_1$ to $w$, then in $\pi$ there are copies $A_0$, $B_0$ of $A$, $B$ rooted at $v_1$ such that the rest of $A_0$ lies to the left of $w$ and the rest of $B_0$ lies to the right. Choosing the `right-most' such copy $B_0$, in $\pi'$ all vertices of $A_0\cup B_0$ lie to the left of $w$ (see Fig \ref{f2}).  Thus the expected number of $T$-jumping edges in a uniformly random ordering of the vertices of $G$ is at least $d-|T|+1$.
\end{proof}

\begin{figure}[ht]
\centering
\begin{tikzpicture}[
  line cap=round,
  line join=round,
  every path/.style={draw=black, line width=0.9pt}
]
  % Left-hand picture
  \begin{scope}
    \coordinate (p) at (0,0);

    % Curved triangle B^*: base width 0.90
    \draw (p)
      .. controls (1.15,1.80) and (3.85,1.95) .. (4.60,0.05);
    \draw (p)
      .. controls (1.05,1.45) and (3.15,1.55) .. (3.70,0.05);

    \draw (3.70,0.05) -- (4.60,0.05);
    \draw[<-] (4.02,0.05) -- (4.28,0.05);
    \node at (4,0.3) {$B$};

    % jumping edge between A and B^*
    \draw (p)
      .. controls (0.95,1.10) and (2.50,1.30) .. (2.85,0.05);
    \fill (2.85,0.05) circle (1.7pt);

    % Curved triangle A: base width 0.60
    \draw (p)
      .. controls (0.75,0.65) and (1.80,0.82) .. (2.25,0.05);
    \draw (p)
      .. controls (0.60,0.30) and (1.35,0.30) .. (1.65,0.05);

    \draw (1.65,0.05) -- (2.25,0.05);
    \node at (1.85,0.22) {$A$};

    \fill (p) circle (1.7pt);
  \end{scope}

  % Central arrow
  \draw[->,line width=1pt] (5.30,0.60) -- (6.30,0.60);

  % Right-hand picture
  \begin{scope}[xshift=7.05cm]
    \coordinate (q) at (0,0);

    % jumping outer edge
    \draw (q)
      .. controls (1.15,1.80) and (3.85,1.95) .. (4.60,0.05);
    \fill (4.60,0.05) circle (1.7pt);

    % Curved triangle A: base width 0.60
    \draw (q)
      .. controls (1.05,1.45) and (3.15,1.55) .. (3.70,0.05);
    \draw (q)
      .. controls (0.98,1.15) and (2.70,1.35) .. (3.10,0.05);

    \draw (3.10,0.05) -- (3.70,0.05);
    \node at (3.35,0.28) {$A$};

    % Curved triangle B: base width 0.90
    \draw (q)
      .. controls (0.80,0.75) and (2.05,0.90) .. (2.55,0.05);
    \draw (q)
      .. controls (0.60,0.30) and (1.35,0.30) .. (1.65,0.05);

    \draw (1.65,0.05) -- (2.55,0.05);
    \draw[->] (1.97,0.05) -- (2.23,0.05);
    \node at (1.95,0.3) {$B$};

    \fill (q) circle (1.7pt);
  \end{scope}
\end{tikzpicture}
\caption{\label{f2} An element of $I_\pi$ becomes part of a copy of $T$ with a jumping edge ($B$ reverses order)}
\end{figure}

The proof above is inspired by Thomas Bloom's presentation of GPT-6 Astra's argument.  However, it is written in (arguably) more natural language, and Case 2 in the proof is simplified, allowing us to use a simple involution on permutations.

\section{Digraphs}\label{sec:digraphs}

An orientation of a graph is {\em antidirected} if it does not contain an oriented path with two edges; in other words, every non-isolated vertex is either a source or a sink.  We think of a single-vertex tree as having two antidirected orientations, one in which the vertex is a source, and one where it is a sink.
%We adjust our definitions to antidirected orientations of a tree.

Let $\T$ be a rooted tree with an antidirected orientation.  Given a digraph $D$ with an ordered vertex set $v_1<\dots<v_n$, we say that an edge $e$ joining $v_1$ and $v_i$ is {\em $\T$-jumping} if 
\begin{itemize}
    \item $G$ contains a copy of $\T$ rooted at $v_1$, with all its vertices in $\{v_1,\dots,v_{i-1}\}$; and
    \item the orientation of $e$ is such that adding $e$ to $\T$ gives an antidirected tree.
\end{itemize}
We allow the case $|\T|=1$, in which case we declare the vertex to be a source or a sink, which determines the orientation for jumping edges.

We will prove the following theorem. 

\begin{theorem}\label{th2}
    Let $D$ be a digraph with average outdegree $d$, and let $\T$ be a rooted tree with an antidirected orientation.  Then the expected number of $\T$-jumping edges in a uniformly random ordering of the vertices of $D$ is at least $d-|\T|+1$.
\end{theorem}

\begin{proof}
The argument is essentially the same as before: for the base case, if $|\T|=1$ and its vertex is declared to be a source then a $\T$-jumping edge in an ordering is just an outedge from the first vertex, and the expected number of these is the average outdegree $d$; similarly, if $|\T|$ is a sink then the expected number of jumping edges is the average indegree $d$.

The two cases now work as in the proof of Theorem \ref{th1} (noting that in Case 2, the quantity $d_G(v_1)$ is replaced by the outdegree or indegree of $v_1$ as appropriate): the tree constructed in both cases remains antidirected. 
\end{proof}

To deduce Conjecture~\ref{AHLRTconj} for an arbitrary tree $\T$ with an antidirected orientation, apply Theorem~\ref{th2} to a tree $\T'$ obtained by removing a leaf from $\T$, and rooting the resulting tree at the neighbour of the removed leaf. If $D$ contains no copy of $\T$, then however its vertices are ordered, there are no $\T'$-jumping edges, so the average outdegree is at most $|\T'|-1=|\T|-2$.

\section{Eulerian digraphs}\label{sec:eulerian}

In this section, we will prove Theorem \ref{eulerianmain}.  The proof follows the same pattern as in earlier sections, except that an extra idea from \cite{mv} is needed in one place: this will require a lemma from \cite{mv}, which we state and prove below.

Let $\T$ be a rooted tree with an arbitrary orientation.  Given a digraph $D$ with an ordered vertex set $v_1<\dots<v_n$, we say that an edge $e$ joining $v_1$ and $v_i$ is {\em $\T$-out-jumping} if 
\begin{itemize}
    \item $G$ contains a copy of $\T$ rooted at $v_1$, with all its vertices in $\{v_1,\dots,v_{i-1}\}$; and
    \item the orientation of $e$ is from $v_1$ to $v_i$.
\end{itemize}
We define {\em $\T$-in-jumping} similarly, except that $e$ must be oriented from $v_i$ to $v_1$.

\begin{figure}[ht]
\centering
\begin{tikzpicture}[
  line cap=round,
  line join=round,
  every path/.style={draw=black, line width=0.9pt},
  forward arrow/.style={
    postaction={decorate},
    decoration={
      markings,
      mark=at position 0.55 with {\arrow{>}}
    }
  },
  backward arrow/.style={
    postaction={decorate},
    decoration={
      markings,
      mark=at position 0.55 with {\arrow{<}}
    }
  }
]
  % Left-hand configuration
  \begin{scope}
    \coordinate (x1) at (0,0);
    \coordinate (t1left) at (2.10,0);
    \coordinate (t1right) at (3.05,0);
    \coordinate (z1) at (4.75,0);

    % Curved triangle T with horizontal base
    \draw (x1)
      .. controls (0.95,1.15) and (2.55,1.20) .. (t1right);
    \draw (x1)
      .. controls (0.70,0.55) and (1.65,0.65) .. (t1left);
    \draw (t1left) -- (t1right);

    \node at (2.50,0.30) {$\T$};

    % Outer directed edge
    \draw[forward arrow]
      (x1) .. controls (1.30,1.85) and (4.05,1.95) .. (z1);

    % End vertices
    \fill (x1) circle (2pt);
    \fill (z1) circle (2pt);
  \end{scope}

  % Right-hand configuration
  \begin{scope}[xshift=6.40cm]
    \coordinate (x2) at (0,0);
    \coordinate (t2left) at (2.10,0);
    \coordinate (t2right) at (3.05,0);
    \coordinate (z2) at (4.75,0);

    % Curved triangle T with horizontal base
    \draw (x2)
      .. controls (0.95,1.15) and (2.55,1.20) .. (t2right);
    \draw (x2)
      .. controls (0.70,0.55) and (1.65,0.65) .. (t2left);
    \draw (t2left) -- (t2right);

    \node at (2.50,0.30) {$\T$};

    % Outer edge directed towards the root
    \draw[backward arrow]
      (x2) .. controls (1.30,1.85) and (4.05,1.95) .. (z2);

    % End vertices
    \fill (x2) circle (2pt);
    \fill (z2) circle (2pt);
  \end{scope}
\end{tikzpicture}
\caption{\label{fe1} $\T$-out-jumping and $\T$-in-jumping edges}
\end{figure}

We prove the following analogue of Theorem \ref{th2}.

\begin{theorem}\label{eulerian}
    Let $D$ be an eulerian digraph with average outdegree $d$, and let $\T$ be a rooted tree.  Then the expected number of $\T$-out-jumping edges in a uniformly random ordering of the vertices of $D$ is at least $d-|\T|+1$; similarly, the expected number of $\T$-in-jumping edges is at least $d-|\T|+1$.
\end{theorem}

Theorem \ref{eulerianmain} is deduced from Theorem \ref{eulerian} in the same way that Conjecture~\ref{AHLRTconj} is deduced from Theorem~\ref{th2}. 

The proof of Theorem \ref{eulerian} requires an appealing lemma from Mubayi and Verstraete \cite{mv}
concerning increasing set systems; recall that a family $\mathcal{F}$ of subsets of $U$ is {\em increasing} if $A\in\mathcal{F}$
and $B\supseteq A$ implies $B\in\mathcal{F}$.

For an ordering $\sigma$ of a finite set $U$ and $u\in U$, let
$S_{<u}(\sigma)$ denote the set of elements strictly before $u$, and 
$S_{\le u}(\sigma)=S_{<u}(\sigma)\cup\{u\}$.  

\begin{lemma}\label{prefix}
Let $J,K\subseteq U$ satisfy $|J|=|K|$, and let
$\mathcal{F}$ be an increasing family of subsets of $U$. If $\sigma$ is a uniformly random ordering
of $U$, then
\begin{equation}\label{v0}
\mathbb{E}\left[
\sum_{u\in J}\mathbf{1}_{\{S_{\le u}(\sigma)\in\mathcal{F}\}}
-
\sum_{w\in K}\mathbf{1}_{\{S_{<w}(\sigma)\in\mathcal{F}\}}
\right]
\le 1.
\end{equation}
\end{lemma}

The GPT-6 Astra proof given in \cite{mv} is relatively complicated; here we give an alternative simple proof, which also has the benefit of establishing the condition for equality (see Proposition~\ref{prefixeq} in the appendix).  The key is the following simple claim.

\begin{claim}\label{Claimuw}
For any distinct $u,w\in U$ we have
\begin{equation}\label{v1}
 \Pr\bb{S_{\le w}(\sigma) \in \cF} \ge \Pr\bb{S_{<u}(\sigma) \in \cF}.
\end{equation}
\end{claim}
\begin{proof}
Let $\sigma'$ be the ordering in which $u$ and $w$ are swapped, so the position of $u$ in $\sigma$ is the position of $w$ in $\sigma'$. Let $R=U\setminus\{u,w\}$. Since all elements of $R$ have the same position in $\sigma$ and in $\sigma'$, we have
\begin{equation}\label{Suw}
 S_{\le w}(\sigma)\cap R = S_{<u}(\sigma')\cap R.
\end{equation}
Since $S_{\le w}(\sigma)$ always contains $w$, and $S_{<u}(\sigma')$ never contains $u$, it follows that $S_{\le w}(\sigma) \supset S_{<u}(\sigma')$. If the latter is in $\cF$, then so is the former, completing the proof.
\end{proof}

\begin{proof}[Proof of Lemma~\ref{prefix}]
We may assume that $\cF$ is non-trivial. Note first that, for any $v$,
\begin{equation}\label{v2}
\mathbb{E}\left[
\mathbf{1}_{\{S_{\le v}(\sigma)\in\mathcal{F}\}}
-
\mathbf{1}_{\{S_{< v}(\sigma)\in\mathcal{F}\}}
\right]
= p_v,
\end{equation}
where $p_v$ is the probability that, as we add elements in the order given by $\sigma$, the point when we add $v$ is the point where we enter $\cF$. For a given $\sigma$, this occurs for exactly one $v$, so $\sum_vp_v=1$.

Now consider the expectation \eqref{v0}.  For $v\in J\cap K$, both sums contain $v$ and by \eqref{v2} the terms containing $v$ contribute $p_v$ to the expectation.  For the remaining terms, as $|J|=|K|$ we can fix a bijection between elements $u\in J\setminus K$ and $w\in K\setminus J$.  For any pair $u,w$, we have (using \eqref{v2} again)
\begin{align*}
    \mathbb{E}\left[
\mathbf{1}_{\{S_{\le u}(\sigma)\in\mathcal{F}\}}
-
\mathbf{1}_{\{S_{< w}(\sigma)\in\mathcal{F}\}}
\right]
=
\mathbb{E}\left[
\mathbf{1}_{\{S_{< u}(\sigma)\in\mathcal{F}\}}
-
\mathbf{1}_{\{S_{\le w}(\sigma)\in\mathcal{F}\}}
\right]
 + p_u + p_w.
\end{align*}
But the Claim~\ref{Claimuw} says exactly that the expectation on the right is at most 0, so 
\begin{align}\label{v3}
    \mathbb{E}\left[
\mathbf{1}_{\{S_{\le u}(\sigma)\in\mathcal{F}\}}
-
\mathbf{1}_{\{S_{< w}(\sigma)\in\mathcal{F}\}}
\right]
\le
p_u + p_w.
\end{align}
Summing \eqref{v2} over $J\cap K$ and \eqref{v3} over the pairs in our bijection between $J\setminus K$ and $K\setminus J$, it follows that the expectation in \eqref{v0} is at most $\sum_{v\in J\cup K}p_v \le 1$.
\end{proof}

\begin{proof}[Proof of Theorem~\ref{eulerian}]
We follow the same proof structure as before, arguing by induction on $|\T|$. The base case, with $|\T|=1$, is immediate.  So we may assume that $\T$ is an oriented tree with root $r$ and at least two vertices, and that we have proved the theorem for trees with fewer vertices.
As before, we split into two cases depending on whether the degree of $r$ is 1 (Case 1) or greater than one (Case 2).  

In Case 2, the proof runs exactly the same as before: to get a lower bound on the expected number of $\T$-out-jumping edges in a uniformly random ordering of the vertices of $D$, we split into trees $\A$ and $\B$, rooted at $r$, and apply the inductive hypothesis to bound the expected number of $\A$-out-jumping edges and the expected number of $\B$-out-jumping edges, completing the argument as before.  For $\T$-in-jumping edges, we consider $\A$-in-jumping and $\B$-in-jumping edges.

For Case 1, we assume that the $r$ is adjacent to a vertex $s$, and deleting $r$ gives a subtree $\A$.  We may assume that the edge joining $r$ and $s$ is directed from $r$ to $s$.

\begin{figure}[ht]
\centering
\begin{tikzpicture}[
  line cap=round,
  line join=round,
  every path/.style={draw=black, line width=0.9pt},
  forward arrow/.style={
    postaction={decorate},
    decoration={
      markings,
      mark=at position 0.5 with {\arrow{>}}
    }
  }
]
  \coordinate (r) at (0,0);
  \coordinate (aleft) at (1.60,0);
  \coordinate (aright) at (2.40,0);
  \coordinate (s) at (3.80,0);

  % Curved copy of A, rooted at s on the right
  \draw (s)
    .. controls (3.10,0.90) and (2.05,1.00) .. (aleft);
  \draw (s)
    .. controls (3.25,0.40) and (2.70,0.50) .. (aright);
  \draw (aleft) -- (aright);

  \node at (2.02,0.28) {$\A$};

  % Curved directed edge from r to s
  \draw[forward arrow]
    (r) .. controls (1.05,1.35) and (3.00,1.45) .. (s);

  % Vertices
  \fill (r) circle (2pt);
  \fill (s) circle (2pt);

  % Labels
  \node[below=3pt] at (r) {$r$};
  \node[below=3pt] at (s) {$s$};
\end{tikzpicture}
\caption{\label{fe2} Orientation of $\T$ in Case 1}
\end{figure}

We say that an edge $e=v_1v_i$ is \emph{$\A$-in-spanning} there is a copy of $\A$ rooted at $v_i$ with vertices contained in $\{v_2,\ldots,v_i\}$ and the edge $e$ is directed from $v_1$ to $v_i$ (and thus into the root of the copy of $\A$).  An edge is $\A$-in-spanning if and only if it is $\A$-{in}-jumping in the order obtained by swapping its vertices.  Arguing as before, we see that the expected number of $\A$-in-spanning edges is the same as the expected number of $\A$-{in}-jumping edges; and, by induction, this expectation is at least $d-|\A|+1=d-|\T|+2$. 

To bound the expected number of $\T$-out-jumping edges, we note (as before) that if we consider the $\A$-in-spanning edges in any ordering, then all but the first is $\T$-out-jumping.  So the expected number of $\T$-out-jumping edges is at least $d-|\T|+1$.

For $\T$-in-jumping edges, we need to do a little more, as having many $\A$-in-spanning edges gives many $\T$-{\em out}-jumping edges but does not imply that we have a $\T$-{\em in}-jumping edge.  So let us fix a vertex $v_1$ and consider orderings starting with $v_1$. We will show that taking the rest of the ordering random, the expected number of $\T$-in-jumping edges is at most $1$ smaller than the expected number of $\A$-in-spanning edges; the result then follows.
    
To see this, let $U=V\setminus\{v_1\}$ be the set we are randomly ordering, let $I$ be the in-neighbours and $O$ the out-neighbours of $v_1$, so $|I|=|O|$.
Call a subset $W$ of $U$ \emph{rich} if $\{v_1\}\cup W$ contains a copy of $\T$ rooted at $v_1$. Clearly the family $\cF$ of rich sets is increasing. For a given ordering $<$, if an out-neighbour $u$ of $v$ is $\A$-in-spanning then $S_{\le u}\in \cF$; it follows that the expected number of out-neighbours $u$ with $S_{\le u}\in \cF$ is at least the expected number of $\A$-in-spanning edges. On the other hand, an in-neighbour $w$ of $v_1$ is $\T$-in-jumping if and only if $S_{<w}\in \cF$.  Lemma \ref{prefix} thus implies that expected number of $\T$-in-jumping edges is at least the expected number of $\A$-in-spanning edges minus~1.

Finally, we remove the conditioning on the first vertex: since the expected number of $\A$-in-spanning edges is at least $d-|\T|+2$, the expected number of $\T$-{out}-jumping edges is at least $d-|\T|+1$.

\end{proof}

\section{Extremal graphs for Erd\H os--S\'os}\label{sec:extremalgraphs}

In this section, we determine the extremal graphs for the Erd\H os--S\'os Conjecture.

\begin{proof}[Proof of Theorem~\ref{th:extremal}]
Suppose that $e(G)=(k-2)|G|/2$ and $T$ is a tree with $k$ vertices.  If $G$ does not contain a copy of $T$ then there are two cases:
\begin{itemize}
    \item If $T=K_{1,k-1}$ then $G$ has no vertex with degree more than $k-2$.  It follows that $G$ is $(k-2)$-regular; and any $(k-2)$-regular graph is extremal.
    \item If $T$ is not a star, then any vertex-disjoint union of copies of $K_{k-1}$ is extremal.  We will show that every extremal graph has this form. 
\end{itemize}

The first bullet is clear; all that remains is to prove that if $T$ is not a star, then any extremal graph is a vertex-disjoint union of copies of $K_{k-1}$.

Let $T'$ be the tree obtained by deleting the leaves of $T$: since $T$ is not a star, $T'$ must have at least two vertices.  Let $x$ be a leaf of $T'$, and let $y$ be the neighbour of $x$ in $T'$.  Let $S$ be the set of leaves of $T$ adjacent to $x$, and let $A$ be the rooted tree obtained from $T$ by deleting $x$ and all leaves of $T$ adjacent to it, and declaring $y$ to be the root (see Fig \ref{f3}). By construction we have $s:=|S|\ge1$ and $|A|=k-s-1\ge2$ (the latter since $y$ and at least one leaf of $T$ belong to $A$).  

\begin{figure}[ht]
\centering
\begin{tikzpicture}[
  line cap=round,
  line join=round,
  every path/.style={draw=black, line width=0.9pt}
]
  % The set S
  \coordinate (s1) at (0,0.65);
  \coordinate (s2) at (0,0);
  \coordinate (s3) at (0,-0.65);
  \coordinate (x)  at (2.15,0);
  \coordinate (y)  at (4.20,0);

  % Edges from S to x
  \draw (s1) -- (x);
  \draw (s2) -- (x);
  \draw (s3) -- (x);

  % Edge xy
  \draw (x) -- (y);

  % Copy of A rooted at y
  \coordinate (a1) at (6.10,1.05);
  \coordinate (a2) at (6.10,-1.05);

  \draw (y) -- (a1);
  \draw (y) -- (a2);
  \draw (a1) -- (a2);

  % Vertices
  \fill (s1) circle (1.8pt);
  \fill (s2) circle (1.8pt);
  \fill (s3) circle (1.8pt);
  \fill (x)  circle (2.2pt);
  \fill (y)  circle (2.2pt);

  % Labels
  \node[above=3pt] at (x) {$x$};
  \node[above=3pt] at (y) {$y$};
  \node at (5.35,0.05) {$A$};

  % Brace and label S
  \draw[decorate,decoration={brace,amplitude=6pt}]
    (-0.25,-0.78) -- (-0.25,0.78);
  \node[left=10pt] at (-0.25,0) {$S$};
\end{tikzpicture}
\caption{\label{f3} The tree $T$}
\end{figure}

We break the proof into three steps.

\medskip
{\em (1) Every ordering has exactly $s$ edges that are $A$-spanning.}
\smallskip

 By Theorem \ref{th1}, the expected number of $A$-jumping edges in a random ordering of the vertices of $G$ is at least $s$.  It follows (as in the proof of Case 1 in Section~\ref{sec:graphs}) that the expected number of $A$-spanning edges is at least $s$.  If there is an ordering with at least $s+1$ $A$-spanning edges, then we obtain a copy of $T$: take any $s+1$ such edges, and add the copy of $A$ guaranteed by the earliest edge.  It follows that every ordering has exactly $s$ edges that are $A$-spanning.

\medskip
{\em (2) $G$ is $(k-2)$-regular. }
\smallskip

For any vertex $v$, consider any ordering with $v$ as first vertex, followed by its neighbours and then the remaining vertices: since $s$ edges are $A$-spanning then the copy of $A$ guaranteed by the earliest of these edges is contained in the neighbourhood of $v$ and is disjoint from the other $A$-spanning edges, so $v$ has at least $|A|+s-1=k-2$ neighbours.  It follows that $G$ is $(k-2)$-regular.

\medskip
{\em (3) Every component of $G$ is a complete graph.}
\smallskip

Suppose not. We will show for a contradiction that $G$ contains a copy of $T$, by constructing a suitable embedding $\phi:V(T)\to V(G)$. 

Since $G$ has a non-complete component it contains vertices $x$, $y$ and $z$ such that $xy$, $yz\in E(G)$ but $xz\not\in E(G)$. Let $v_1v_2\cdots$ be a longest path in $T$, which has at least $4$ vertices, since $T$ is not a star. Set $\phi(v_2)=x$, $\phi(v_3)=y$ and $\phi(v_4)=z$. Now continue greedily step-by-step, each time embedding a vertex $u$ of $T$ with a (necessarily unique) already embedded neighbour $v$. At each step, the only constraint is that $\phi(u)$ should be an unused neighbour of $\phi(v)$. Since $\phi(v)$ has $k-2$ neighbours this is always possible, unless $u$ is the last vertex to be embedded, and all vertices other than $u$, $v$ are mapped to neighbours of $\phi(v)$. To avoid this, we simply embed $v_1$ last (which is possible since it's a leaf), noting that $\phi(v_4)=z$ is not a neighbour of $\phi(v_2)=x$.
\end{proof}

\section{Extremal digraphs}\label{sec:extremaldigraphs}

We now turn to the extremal digraphs for the conjecture of Addario-Berry, Havet, Linhares Sales, Reed and Thomass\'e.  With a little more work, we are able to determine these.

\begin{proof}[Proof of Theorem~\ref{th:extremal2}]

Now suppose that $D$ is a digraph with $n$ vertices and $(k-2)n$ edges, and $\T$ is an antidirected tree with $k$ vertices.  If $D$ does not contain a copy of $\T$ then there are two cases:
\begin{itemize}
    \item If $\T$ is a star then the antidirected condition implies that the centre of the star is either a source or a sink.  In the first case, any $(k-2)$-out-regular digraph is extremal; and in the second, any $(k-2)$-in-regular digraph is extremal.  It is clear that these are the only extremal digraphs.
    \item If $\T$ is not a star, then any vertex-disjoint union of copies of complete digraphs on $k-1$ vertices, with all edges directed in both directions, is extremal.  We will show that every extremal digraph has this form. 
\end{itemize}

In the graph case, the key step was to show that $G$ is regular. In the digraph case, we need biregularity. We start by showing that this is sufficient.

\medskip
{\em (1) It suffices to show that $D$ is both in- and out-regular. }
\smallskip

Suppose that $D$ is bi-regular, so every vertex has in-degree and out-degree equal to $k-2$. Let $v_1v_2v_3v_4$ be a path in $\T$ with $v_1$ a leaf (which exists since $\T$ is not a star), and without loss of generality suppose the edges are directed $v_1\leftarrow v_2\to v_3\leftarrow v_4$. Suppose first that $D$ contains a vertex $y$ whose in-neighbourhood is not bi-complete, so $y$ has in-neighbours $x$ and $z$ with $x\not\to z$. Set $\phi(v_2)=x$, $\phi(v_3)=y$ and $\phi(v_4)=z$, and embed greedily as before. When embedding $u$ with a unique already-embedded neighbour $v$, we need either an unused in- or unused out-neighbour of $\phi(v)$. The only possible problem is at the last step, but we take $v_1$ last, and since $\phi(v_4)$ is not an out-neighbour of $\phi(v_2)$, there is at least one unused out-neighbour of $\phi(v_2)$ to map $v_1$ to.

Thus we may assume that every in-neighbourhood is bi-complete. But this easily implies that each component of $D$ is bi-complete. Let $S$ be the in-neighbourhood of some vertex $v$, so $|S|=k-2$. Then every vertex in $S$ has in- and out- degree $k-3$ within $S$, so exactly one in-edge from outside $S$. If all such edges come from $v$, then the component containing $v$ is bicomplete. Otherwise, there is an edge $x\to s$ for some $s\in S$ and $x\notin S\cup \{v\}$. There is at least one other vertex $t\in S$ (since $k\ge 4$ since $\T$ is not a star), and now by bi-completeness of the in-neighbourhood of $s$, the edge $t\to x$ is present. But then the out-degree of $t$ is at least $k-1$, a contradiction.

\bigskip

Label the vertices of $\T$ as in Fig \ref{f3}: without loss of generality, we assume that $x$ is a source, with $s\ge1$ adjacent leaves, and that $\A$ is an oriented tree with at least two vertices.  Note that the edge $xy$ is directed from $x$ to $y$, and that the leaves adjacent to $x$ are sinks (see Fig \ref{f4}). We now proceed in several steps, recalling that in this context, for an edge joining $v_1$ to $v_i$ to be $\A$-spanning or $\A$-jumping, its orientation must be towards the root of $\A$.

\begin{figure}[ht]
\centering
\begin{tikzpicture}[
  line cap=round,
  line join=round,
  every path/.style={draw=black, line width=0.9pt},
  mid arrow/.style={
    postaction={decorate},
    decoration={
      markings,
      mark=at position 0.5 with {\arrow{>}}
    }
  }
]
  % The set S
  \coordinate (s1) at (0,0.65);
  \coordinate (s2) at (0,0);
  \coordinate (s3) at (0,-0.65);
  \coordinate (x)  at (2.15,0);
  \coordinate (y)  at (4.20,0);

  % Directed edges out of x, with arrows at their midpoints
  \draw[mid arrow] (x) -- (s1);
  \draw[mid arrow] (x) -- (s2);
  \draw[mid arrow] (x) -- (s3);
  \draw[mid arrow] (x) -- (y);

  % Copy of A rooted at y
  \coordinate (a1) at (6.10,1.05);
  \coordinate (a2) at (6.10,-1.05);

  \draw (y) -- (a1);
  \draw (y) -- (a2);
  \draw (a1) -- (a2);

  % Vertices
  \fill (s1) circle (1.8pt);
  \fill (s2) circle (1.8pt);
  \fill (s3) circle (1.8pt);
  \fill (x)  circle (2.2pt);
  \fill (y)  circle (2.2pt);

  % Labels
  \node[above=3pt] at (x) {$x$};
  \node[above=3pt] at (y) {$y$};
  \node at (5.35,0.05) {$\A$};

  % Brace and label S
  \draw[
    decorate,
    decoration={brace,amplitude=6pt}
  ] (-0.25,-0.78) -- (-0.25,0.78);
  \node[left=10pt] at (-0.25,0) {$S$};
\end{tikzpicture}
\caption{\label{f4} The tree $\T$}
\end{figure}

\medskip
{\em (2) Every ordering has exactly $s$ edges that are $\A$-spanning.}
\smallskip

Consider first the subtree $\A$ rooted at $y$, noting that $y$ is a sink.  As in the graph case, by induction the expected number of $\A$-jumping edges is at least $s$, and so the expected number of $\A$-spanning edges is at least $s$; if there are at least $s+1$ $\A$-spanning edges then we obtain a (correctly oriented) copy of $\T$.  It follows that every ordering has exactly $s$ $\A$-spanning edges.  

\medskip
{\em (3) $D$ is $(k-2)$-out-regular. }
\smallskip

For any vertex $v$, consider any ordering with $v$ as first vertex, followed by its out-neighbours and then the remaining vertices: since $s$ edges are $\A$-spanning the copy of $\A$ guaranteed by the earliest of these edges is contained in the out-neighbourhood $\Gamma^+(v)$ of $v$ and is disjoint from the other $\A$-spanning edges, so $v$ has at least $|\A|+s-1=k-2$ out-neighbours.  Since this holds for every vertex, it follows that $D$ is $(k-2)$-out-regular.

\medskip
{\em (4) We may suppose that $\A$ contains vertices $a, b $ such that $a$ is a source and $b$ is a leaf of $\T$ adjacent to $a$.}
\smallskip

As in the undirected case, consider the tree $\T'$ obtained from $\T$ by deleting all its leaves.  Then $\T'$ has at least two leaves, one of which is $x$; let $a$ be another leaf.  If $a$ is a sink, then using $a$ in place of $x$, we get a decomposition of $\T$ as in Fig \ref{f4}, except with a sink at $a$ in place of the source at $x$.  Arguing as above, we get that $D$ is $(k-2)$-in-regular; by (1) and (3), this is enough.  Otherwise, $a$ is a source. Since $a$ is a vertex of $\T'$, there is some leaf $b$ of $\T$ adjacent to $a$, and $a$, $b$ both belong to $\A$.

\medskip
{\em (5) For every vertex $v$, every set $W$ of $|\A|=k-s-1$ out-neighbours of $v$, and every $z$ in $W$, there is a copy of $\A$ rooted at $z$ and with vertex set $W$.}
\smallskip

As in (3), consider an ordering with first vertex $v$, followed by its out-neighbours $w_1,\dots,w_{k-2}$ and then the remaining vertices.  If $vw_i$ is $\A$-spanning then $i\ge|\A|=k-s-1$ (as we need to fit in a copy of $\A$); and since there are exactly $s$ $\A$-spanning edges, it follows that $vw_i$ is $\A$-spanning for all $i\ge k-s-1$.  In particular, there is a copy of $\A$ rooted at $w_{k-s-1}$ and with vertex set $\{w_1,\dots,w_{k-s-1}\}$.  Since we can take the $w_i$ in any order, the claim follows.

\medskip
{\em (6) Suppose that $u,v$ have a common out-neighbour $z$, and $u$ and $v$ are not bi-adjacent. Then $\Gamma^+(u)=\Gamma^+(v)$, and in particular $u$ and $v$ are non-adjacent.}
\smallskip

We may suppose without loss of generality that $u\not\to v$. Suppose that $\Gamma^+(u)\ne\Gamma^+(v)$. Then by out-regularity there is some $w\in \Gamma^+(u)\setminus \Gamma^+(v)$. By (5), since $|\A|\ge 2$, there is a copy $\Azero$ of $\A$ contained in $\Gamma^+(u)$ that includes the vertex $w$ and is rooted at $z$. By assumption, $\Azero$ does not contain $v$. Furthermore, since $w\notin \Gamma^+(v)$, the vertex set of $\Azero$ includes at most $|\A|-1=k-2-s$ vertices from $\Gamma^+(v)$. Then adding $v$ and $s$ further out-neighbours of $v$ to $\Azero$ gives a copy of $\T$, a contradiction.

\medskip
{\em (7) For every vertex $u$ there is $v\in \Gamma^+(u)$ such that $u\leftrightarrow v$ and their out-neighbourhoods are equal (apart from each other).} 
\smallskip

By (5) there is a copy $\Azero$ of $\A$ in $\Gamma^+(u)$. Now by (4), $\A$ contains a source $a$ and a leaf $b$ of $\T$ adjacent to $a$ (so $b$ is not the root of $\A$). Let $v$, $w$ be the images of $a$ and $b$ in $\Azero$. Suppose $v$ has an out-neighbour $w'$ outside $\Gamma^+(u)\cup\{u\}$. Then we may replace $w$ by $w'$ to obtain a new copy of $\A$, and then add $u$, $w$, and the $s-1$ vertices of $\Gamma^+(u)$ outside $\Azero$ to make a copy of $\T$. Hence $\Gamma^+(v)\subset \Gamma^+(u)\cup\{u\}\setminus\{v\}$, a set of size $k-2$. Thus $\Gamma^+(v)$ is equal to this set.

\medskip
{\em (8) $D$ is in-regular. }
\smallskip

If not, then since the average in-degree is $k-2$, there is some vertex $v$ whose in-neighbourhood $I$ has $|I|\ge k-1$. If $I$ is an independent set, then by (6) all vertices in $I$ have the same set $J$ of out-neighbours which is necessarily disjoint from $I$. We thus have a complete bipartite graph with both vertex classes of size at least $k-2$ and all edges directed from the first class to the second class.  As $\T$ is not a star, this embeds $\T$. We may thus suppose that $I$ is not independent.

By (6) each pair in $I$ is either non-adjacent or bi-adjacent. Moreover, if $u,v$, and $v,w$ are non-adjacent, then (applying (6) to $u$, $v$, which shows that $w\notin \Gamma^+(u)$), so are $u,w$. Hence the subgraph induced by $I$ is complete multipartite with all edges bidirectional. Since $I$ is not independent, there are at least two vertex classes.

Now by (7), every vertex is in a bi-directed edge, so there is $u\in I$ with $v\to u$. Suppose $w$ is in a different class within $I$, and that $v\not\to w$. Since $u\leftrightarrow w$, $v$ and $w$ have a common out-neighbour $u$, so we obtain a contradiction to (6). Hence $v$ sends an edge to every vertex not in the same class as $u$. But applying this once more with $w$ in place of $u$, $v$ sends and edge to every vertex in $I$, and so has out-degree at least $k-1$, a contradiction.

\bigskip

We have shown that $D$ is both in- and out-regular, and so we are done by (1).
\end{proof}

\section{Extremal eulerian digraphs}\label{sec:extremaleulerian}

Finally, we prove Theorem \ref{th:extremal3}, which determines the extremal digraphs for Theorem~\ref{eulerianmain}.

\begin{proof}[Proof of Theorem \ref{th:extremal3}]
Suppose that $\T$ is a $k$-vertex tree with an arbitrary orientation, and $D$ is an eulerian digraph with average in- (and hence out-)degree $k-2$, and that $D$ contains no copy of $\T$.

First suppose that $\T$ is a star. If $D$ contains any vertex of in-/out-degree at least $k-1$, then it is trivial to embed $\T$ centered at this vertex. Thus $D$ must be $(k-2)$-bi-regular. If $\T$ is antidirected, so its central vertex is a source or a sink, this is the only condition on $D$. Otherwise, if any vertex has $\Gamma^+(v)\ne \Gamma^-(v)$ it is again trivial to embed $\T$, so $D$ is the bi-directed version of a $(k-2)$-regular undirected graph. All such digraphs contain no copies of $\T$ and are hence extremal.

Thus we may suppose that $\T$ is not a star. We also suppose that $D$ contains at least one unidirectional edge, as otherwise we are done by the undirected graph case.

As before, let $x$ be a leaf of the tree obtained by deleting all leaves of $\T$, and define $S$, $y$ and $\A$ as in Fig \ref{f4}, except that now we do not know the orientation of the relevant edges.

\medskip\noindent
{\bf Case (i)}: $x$ is a source or a sink in $\T$. Without loss of generality we may assume that $x$ is a source, as in Fig \ref{f4}. Now arguing exactly as for steps (2) and (3) in the previous section, we see that in every ordering exactly $s=|S|$ edges are $\A$-in-spanning, and then that $D$ is $(k-2)$-out-regular, and hence $(k-2)$-in-regular. We now construct an embedding $\phi$ of  $\T$ into $D$ vertex-by-vertex as in step (1) in the previous section. The only possible problem is embedding the final vertex, which we take to be a leaf $s\in S$. To ensure that we can embed $s$, it suffices to make sure that some $v\notin\{x,s\}$ is mapped by $\phi$ to a vertex not in $\Gamma^+(\phi(x))$.

Let $z$ be a neighbour of $y$ in $\A$. We consider two cases: if the edge $yz$ in $\T$ is directed $z\to y$ then we are done by the argument in (1) in the previous section (with $s\leftarrow x\to y\leftarrow z$ playing the role of $v_1\leftarrow v_2\to v_3\leftarrow v_4$). Otherwise we have a path $s\leftarrow x\to y\to z$ in $\T$. Let $a$ and $b$ be distinct vertices of $D$ with $a\to b$ and $b\not\to a$, which exist by assumption. Since $b\in\Gamma^+(a)\setminus\Gamma^+(b)$, by out-regularity there is some $c\in \Gamma^+(b)\setminus \Gamma^+(a)$. Since $b\not\to a$, $c\ne a$, so we have found distinct vertices $a,b,c$ of $D$ such that $a\to b$, $b\to c$ but $a\not\to c$. Now we start our embedding by setting $\phi(x)=a$, $\phi(y)=b$ and $\phi(z)=c$. Since $c\notin \Gamma^+(a)$ this meets the requirement above, with $v=z$, and we can find a copy of $\T$, a contradiction.

\medskip\noindent
{\bf Case (ii):} $x$ is neither a source nor a sink in $\T$. Without loss of generality we may assume that the edge $xy$ is directed from $x$ to $y$. Thus there is some $s\in S$ with $sx$ directed from $s$ to $x$. From now on let us write $\Tzero$ for $\T$, to avoid confusion with various trees $\T$ considered in the inductive argument.

We start with a simple claim: if $\T$ is a single edge rooted at the sink and $D$ contains some unidirectional edge, then the expected number of $\T$-out-jumping edges is strictly more than $d-|\T|+1=d-1$, where $d=k-2$ is the average degree of $D$. To see this, consider orders with some particular vertex $v_1$ first, with in- and out-degree $t$, taking the rest of the order random. Consider any $u\in \Gamma^+(v_1)$. Then $u$ fails to be $\T$-out-jumping if and only if no $w\in \Gamma^-(v_1)$ strictly preceeds $u$, i.e., $u$ is the first element in the set $\{u\}\cup\Gamma^-(v_1)$. This event has probability $1/t$ if $u\in \Gamma^-(v_1)$ and probability $1/(t+1)$ otherwise. It follows that the expected number of $\T$-out-jumping edges in orders starting with $v_1$ is at least $d^+(v_1)-1$, with strict equality if there is some $u\in \Gamma^+(v_1)\setminus \Gamma^-(v_1)$. Averaging over $v_1$, the claim follows.

At this point we need to consider the details of how the inductive bound on in- and out-jumping edges works. By an \emph{augmented tree} we mean a rooted tree $\Tstar$ with a distinguished leaf $\ell$ adjacent to the root $r$. We call $\ell$ the \emph{extra vertex} of $\Tstar$ and $r\ell$ the \emph{extra edge}. An edge $v_1v_i$ in a particular ordering of the vertices of $D$ is \emph{a $\Tstar$ edge} if there is a copy of $\Tstar$ contained in $\{v_1,\ldots,v_i\}$ with root $v_1$ and with $v_1v_i$ playing the role of the extra edge $e$.
Thus a $\T$-in-jumping edge is exactly a  $\Tstar$ edge where $\Tstar$ is formed from $\T$ by adding a new edge $e$ directed to the root of $\T$, and similarly for $\T$-out-jumping, reversing the direction of $e$. In this language, Theorem~\ref{eulerian} asserts lower bounds on the expected numbers of $\Tstar$ edges for all augmented trees $\Tstar$.
Note that this is simply a restatement of what we have already proved; here, unlike in previous sections, it is convenient to argue with the tree \emph{including} the extra/jumping edge.

Recall that we assume $D$ contains some unidirectional edge.

We say that a leaf $v$ in an augmented tree $\Tstar$ is \emph{special} if the unique path $v=v_1v_2v_3\cdots$ from $v$ to the extra vertex contains at least two edges, and the first two edges are directed consistently, i.e., $v_1\to v_2\to v_3$ or $v_1\leftarrow v_2\leftarrow v_3$.
We claim that we have \emph{strict} inequality in Theorem~\ref{eulerian} whenever our augmented tree $\Tstar$ contains a special leaf. We prove this by induction on $|\Tstar|$, noting that any $\Tstar$ containing a special leaf has $|\Tstar|\ge 3$.

The base case $|\Tstar|=3$ is the claim above: $\Tstar$ consists only of the path $v_1\to v_2\to v_3$ (or the same with arrows reversed), with $v_2v_3$ the extra edge, so $\Tstar$ edges are just $\T$-out-jumping edges where $\T$ is the edge $v_1\to v_2$ rooted at $v_2$. 

For the induction step, suppose $|\Tstar|\ge 4$. In Case 2 of the proof of Theorem~\ref{eulerian}, as before we bound the number of $\T$-out-jumping edges by considering $\A$-out-jumping and $\B$-out-jumping edges, and similarly for in-jumping. The corresponding augmented trees $\Tstar$, $\Astar$ and $\Bstar$ are simply $T$, $A$ and $B$ with an extra distinguised edge from their root. If $\Tstar$ contains a special leaf $v$, then without loss of generality $\Astar$ contains $v$. Since the entire path from $v$ to the extra vertex is also contained in $\Astar$, then $v$ is also special in $\Astar$, so by induction we get strict inequality for $\Astar$ and thus for $\Tstar$.

In Case 1 of the induction step,
recall that the root $r$ of $\T$ is joined to a single vertex $s$ in $\T$, with the edge oriented $r\to s$ (without loss of generality); we consider both $\T$-in-jumping and $\T$-out-jumping edges, corresponding to the two possible $\Tstar$ extending $\T$.
Since $\A$ is defined by removing the root $r$ of $\T$, now $\Astar$ is simply $\T$ with the root moved and the edge $r\to s$ declared the extra edge (corresponding to considering $\A$-in-jumping edges). Now by assumption $\Tstar$ contains a special leaf $v$. The path $P'$ from $v$ to the extra vertex in $\Astar$ is simply the corresponding path $P$ in $\Tstar$ with the final vertex removed. If $v$ is not special, it follows that \red{$P$} contains only three vertices, so $P=vrx$ where $r$ is the root of $\T$ and $\Tstar$ and $x$ is the extra vertex of $\Tstar$. But since $v$ is a leaf and $r$ has degree 1 in this case, then $\T$ is just a single edge, contradicting our assumption that $|\Tstar|\ge 4$. Thus we may assume that $\Astar$ contains a special leaf.

As before the expected numbers of $\A$-in-spanning and $\A$-in-jumping edges are equal, and the expected number of $\T$-out-jumping or $\T$-in-jumping edges is at most one less. But now we have strict inequality for $\A$-in-jumping by induction, and we are done.

Finally, we deduce the presence of a copy of $\Tzero$ by removing a leaf from $\Tzero$, choosing one that is not in $S$, and so is in $\A$. Then the corresponding augmented tree contains the path $s\to x\to y$, so $s$ is a special leaf and we obtain strict inequality in Theorem~\ref{eulerian} and hence a copy of $\Tzero$.    
\end{proof}

\appendix
\section{Equality in Lemma~\ref{prefix}}

In this appendix, we characterize the cases when Lemma \ref{prefix} holds with equality: it turns out that this occurs precisely when the increasing family $\cF$ is generated by a family of subsets of $J\cap K$.

\begin{proposition}\label{prefixeq}
Equality holds in Lemma~\ref{prefix} if and only if $\cF$ is non-trivial and every minimal element of $\cF$ is a subset of $J\cap K$.
\end{proposition}

\begin{proof}
We may suppose that $\cF$ is non-trivial, otherwise the left-hand side of \eqref{v0} is 0 and equality does not hold.

The proof of Lemma~\ref{prefix} proceeded by summing the inequality \eqref{v3} over pairs $u\in J\setminus K$ and $w\in K\setminus J$, and combining this with the inequality $\sum_{v\in J\cup K}p_v \le \sum_{v\in U}p_v=1$.
We thus have equality if and only if (i) equality holds in Claim \ref{Claimuw} for all $u\in J\setminus K$ and $w\in K\setminus J$, and (ii) $\sum_{v\in J\cup K} p_v=1$. Call an element $u\in U$ \emph{irrelevant} if no minimal element of $\cF$ contains $u$, so $R\in \cF$ if and only if $R\cup\{u\}\in \cF$. Now (i) holds if and only if for any $u\in J\setminus K$, $w\in K\setminus J$ and any $\sigma$ we have $S_{\le w}(\sigma)\in \cF$ if and only if $S_{<u}(\sigma')\in \cF$, where $\sigma'$ is the order obtained by swapping $u$ and $w$. Recalling \eqref{Suw}, this certainly holds if $u$ and $w$ are irrelevant. Conversely, if $S_{\le w}(\sigma)\in \cF$ if and only if $S_{<u}(\sigma')\in \cF$, for any $R\subset U\setminus\{u,w\}$ we may choose $\sigma$ to start with $R$, then $w$, then $u$. Then $S_{\le w}(\sigma)=R\cup\{w\}$ and $S_{<u}(\sigma')=R$. Swapping $u$ and $w$ in the order, we obtain $S_{\le w}(\sigma)=R\cup \{u,w\}$ and $S_{<u}(\sigma')=R\cup\{w\}$, so $R\in \cF$ if and only if $R\cup\{u,w\}\in \cF$, which implies that $u$ and $w$ are irrelevant. Hence (i) holds exactly when all elements of $(J\setminus K)\cup (K\setminus J)$ are irrelevant. On the other hand, (ii) holds exactly when all $v\in U\setminus (J\cup K)$ are irrelevant, completing the proof.
\end{proof}
\bigskip

\noindent{{\bf AI declaration:} ChatGPT was used to help draw pictures.}


\begin{thebibliography}{9}

\bibitem{directed}
L. Addario-Berry, F. Havet, C. Linhares Sales, B. Reed, and S. Thomass\'e, Oriented trees in digraphs, {\em Discrete Mathematics} {\bf 313} (2013), 967--974

\bibitem{BP-SS} 
 G. Besomi, M. Pavez-Sign{\'e} and M. Stein,
 On the {E}rd{\H{o}}s--{S}{\'o}s conjecture for trees with bounded degree,
 \emph{Combin. Probab. Comput.} {\bf 30} (2021), 741--761.

 \bibitem{bloom}
T. Bloom, Proof exposition of Erd\H os problem 548,\\ \texttt{https://www.erdosproblems.com/forum/thread/548\#proof-exposition-9}  (2026)

\bibitem{BS}
B. Bollob\'as and A. D. Scott, 
A proof of a conjecture of Bondy concerning paths in weighted digraphs, 
\emph{J. Combin. Theory Ser. B} {\bf 66} (1996), 283--292.

 \bibitem{BD}
 S. Brandt and E. Dobson,
 The Erd\H{o}s--S\'os conjecture for graphs of girth 5,
 \emph{Discrete Math.} {\bf 150} (1996), 411--414.

\bibitem{erdossos} P. Erd\H os, Extremal problems in graph theory, {\em in} Theory of Graphs and Its Applications, Proc. Sympos. Smolenice, pages 29–36, 1964.

\bibitem{chatgpt}
GPT-6 Astra, A counting proof for Erd\H os problem 548,\\ \texttt{https://www.erdosproblems.com/static/548copy.pdf} (2026)

\bibitem{huang}
H. Huang, J. Ma, A. Shapira, B. Sudakov, and R. Yuster, 
Large feedback arc sets, high minimum degree subgraphs, and long cycles in Eulerian digraphs,
{\em Combinatorics, Probability and Computing}, 
{\bf 22} (2013), 859--873.

\bibitem{janzer}
O. Janzer, B. Sudakov, and I. Tomon, 
Long directed paths in Eulerian digraphs,
\texttt{https://arXiv:2101.11601}

\bibitem{knierim}
C. Knierim, M. Larcher, and A. Martinsson, Note on Long Paths in Eulerian Digraphs,
{\em The Electronic Journal of Combinatorics} {\bf 28} (2021), \#P2.37

\bibitem{lean} Lean 4 Erd\H{o}s problem \#548 ({E}rd{\H{o}}s--{S}ós conjecture), Palomar submission repository,
\texttt{https://github.com/tadamcz/erdos548} (2026)

 \bibitem{McLennan}
 A. McLennan,
 The Erd\H os--S\'os conjecture for trees of diameter four,
 \emph{J. Graph Theory} {\bf 49} (2005), 291--301.

\bibitem{mv}
D. Mubayi and J. Verstraete,
Erd\H os--S\'os for digraphs,
\texttt{https://arxiv.org/abs/2609.10987}

\bibitem{pokrovskiy}
 A. Pokrovskiy, 
 Hyperstability in the Erd{\H{o}}s--S{\'o}s conjecture,
 Preprint, arXiv:2409.15191.

\bibitem{rohzon}
 V. Rozhoň, 
 A Local Approach to the {E}rd{\H{o}}s--{S}ós Conjecture,
 \emph{SIAM J. Disc. Math.} {\bf 33} (2019), 643--664.

 \bibitem{maya}
G. Santos, M. Stein and E. Williams, Are trees really just butterflies in disguise?,
\texttt{https://arxiv.org/abs/2609.09142}

\bibitem{Stein}
 M. Stein,
 Tree containment and degree conditions,
  in Discrete Mathematics and Applications (A. M. Raigorodskii and M. Th. Rassias, eds.), Springer Optimization and Its Applications, vol. 165, Springer, Cham, 2020, 459--486. doi:\text{10.1007/978-3-030-55857-4\_19} 

\bibitem{ST-N25}
 M. Stein and A. Trujillo-Negrete,
 Antidirected Trees in Dense Digraphs,
 \emph{SIAM J. Disc. Math.} {\bf 39} (2025), 698--727.

\bibitem{SZ-G24} 
 M. Stein and C. Zárate-Guerén,
 Antidirected subgraphs of oriented graphs,
 \emph{Combin. Probab. Comput.} {\bf 33} (2024), 446--466.
 
\end{thebibliography}
\end{document}